\documentclass[two side,reqno,12pt]{amsart}
\usepackage{bigints}
\usepackage{amssymb,latexsym,xy,eucal}
\usepackage{amssymb,amsmath,graphicx,colortbl,enumerate}
\usepackage{longtable,mathrsfs}
\numberwithin{equation}{section}
\newtheorem{thm}{Theorem}[section]
\newtheorem{lem}[thm]{Lemma}

\theoremstyle{definition}
\newtheorem{example}[thm]{Example}
\newtheorem{rem}{Remark}[section]

\allowdisplaybreaks
\DeclareMathSizes{10}{9}{5}{5}   % For size 11 text
\begin{document}
\baselineskip 18truept
\title[Some New  Retarded Nonlinear Integral Inequalities]{New Retarded Nonlinear Integral Inequalities of the Gronwall-Bellman-Pachpatte Type and Their Applications}
\date{}
\author{Nagesh Kale}

\address{Nagesh Kale, Department of Mathematics,Sant Rawool Maharaj Mahavidyalaya Kudal,
Sindhudurga-416520, India.(M.S.)} \email{nageshkale7991@gmail.com}


\maketitle

\begin{abstract}
The goal of the present article is to offer a number of new retarded nonlinear inequalities of Gronwall, Bellman and Pachpatte kind for a class of integral and integro-differential equations. These inequalities generalize and provide new formulations of some well-known results in the mathematical framework of integral and differential inequalities that have been derived currently as well as in earlier times. These results can be utilized to investigate diverse aspects, both qualitative and quantitative, of a class of aforementioned equations. We propose a few applications to ensure effectiveness of these inequalities.
\end{abstract}

\noindent{ \small \textbf{Keywords}:   integral inequality, retarded integral equations, Gronwall-Bellman-Pachpatte type.
%}
\section{\textbf{Introduction}}
\par In the realm of contemporary advances in several disciplines of mathematics, the quest of integral equations, differential equations, and integro-differential equations has an essential role due to its widespread recognition as a leading instrument of applied research. Evidently, the study of a number of qualitative and quantitative properties of these classes of equations has focused heavily on inequality technique.The extensive literature illuminating the these tools and their evolution can be unveiled form \cite{Pachpatte1998, Bainov1992, Pachpatte2006}, and references specified therein. 
\par In 1919, while researching the dependent nature of systems of differential equations with relative parameters, Gronwall devised the widely recognized integral inequality \cite{Gronwall1919}, which states that 
\begin{thm}[Gronwall \cite{Gronwall1919}]\label{i4} If
\begin{align*}
	0\leq x(\delta)\leq \int\limits_{\mathfrak{c}}^{\delta} \Big(\mathfrak{h}_1x(\mathfrak{\tilde{\delta}})+\mathfrak{h}_2\Big) d\mathfrak{\tilde{\delta}},~~\text{for}~ \delta\in [\mathfrak{c}, \mathfrak{c}+\mathfrak{h}], \mathfrak{h}_1, \mathfrak{h}_2\in\mathbb{R}_+, 
\end{align*} 
for continuous function $x(\delta)$ on $[\mathfrak{c}, \mathfrak{c}+\mathfrak{h}]$ then
\begin{align*}
	0\leq x(\delta)\leq \mathfrak{h}_2\mathfrak{h}e^{\mathfrak{h}_1\mathfrak{h}},~~ \text{for}~\delta\in [\mathfrak{c}, \mathfrak{c}+\mathfrak{h}].
\end{align*}
\end{thm}
Subsequently, Bellman (1943) proposed an intriguing extension of Gronwall's inequality, which reads as
\begin{thm}[Bellman \cite{Pachpatte1998}]\label{i5}
	If 
	\begin{align*}
		0\leq x(\delta)\leq \mathfrak{h}+\int_{\mathfrak{h}_1}^\delta \mathfrak{w}(\tilde{\delta})x(\tilde{\delta}) d\tilde{\delta},~~\text{for}~\delta\in \mathcal{J}=[\mathfrak{h}_1,\mathfrak{h}_2]
	\end{align*} 
for a continuous function $x(\delta)$ and $\mathfrak{h}_1\in\mathbb{R}_+,$ then
	\begin{align*}
	0\leq x(\delta)\leq \mathfrak{h}\exp \left(\int_{\mathfrak{h}_1}^\delta \mathfrak{w}(\tilde{\delta}) d\tilde{\delta}\right),~~\text{for}~\delta\in \mathcal{J}=[\mathfrak{h}_1,\mathfrak{h}_2].
\end{align*} 
\end{thm}
Furthermore, Pachpatte produced a more broad variant of the Gronwall-Bellaman inequality in 1973, which asserts that
\begin{thm}[Pachpatte \cite{Pachpatte1973}]\label{i6}
	If $x, \mathfrak{w}, \tilde{\mathfrak{w}}$ are nonnegative continuous functions defined on $\mathbb{R}_+$ such that
	\begin{align*}
		x(\delta)\leq  \mathfrak{h}+\int_{0}^\delta \mathfrak{w}(\tilde{\delta})x(\tilde{\delta}) d\tilde{\delta}+\int_{0}^\delta \mathfrak{w}(\tilde{\delta})\left(\int_{0}^s \tilde{\mathfrak{w}}(\sigma)x(\sigma) d\sigma\right) d\tilde{\delta},~~\text{for}~\delta\in \mathbb{R}_+,
	\end{align*}
for nonnegative and continuous functions $x, \mathfrak{w}, \tilde{\mathfrak{w}}$ and $\mathfrak{h}\in\mathbb{R}_+$ then
	\begin{align*}
	x(\delta)\leq  \mathfrak{h}\left[1+\int_{0}^\delta \mathfrak{w}(\tilde{\delta})\exp\left(\int_{0}^s [\mathfrak{w}(\sigma)+\tilde{\mathfrak{w}}(\sigma)]d\sigma\right) d\tilde{\delta}\right],~~\text{for}~\delta\in \mathbb{R}_+.
\end{align*}
\end{thm}
In the past few decades, a number of generalizations and extensions of these types of inequalities and their extended discrete analogues have been published \cite{Agarwal2005,Kendre2020,Kendre2021,Bainov1992,Pachpatte2006,Kale2021}. The retarded integral inequalities, which have their roots in the aforementioned integral inequalities, were recently devised by A Shakoor, Wang, Abdeldaim, Yakout, and El-Deeb \cite{Shakoor2023,Abdeldaim2015,Abdeldaim2011,Shakoor2019,Shakoor2020,Wang2012}.  The most recent generalized improvements of a few previous retarded integral inequalities were reported by A. Shakoor et al. \cite{Shakoor2023}. Here, we mention one of the inequalities reported by A. Shakoor et al., which stated that 
\begin{thm}[Shaknoor \cite{Shakoor2023}]\label{i1}
	If
	\begin{align*}
	x'(\delta) \leq l(\delta)+\int_0^{\alpha(\delta)} g_1(\tilde{\delta}) x(\tilde{\delta}) d \tilde{\delta}+\int_0^{\alpha(\delta)} g_2(\tilde{\delta})\left(x^{\lambda_1}(\tilde{\delta})+\int_0^{\tilde{\delta}} g_3(\mu) x^{\lambda_2}(\mu) d \mu\right)^{\frac{1}{\lambda_1}} d \tilde{\delta} \quad \forall \delta \in \mathbb{R}_+,
\end{align*}
for $\lambda_1>\lambda_2 \geq 0$, nonnegative $x, x', g_1, g_2, g_3 \in \textnormal{Cf}_{\mathbb{R}_+}$ and nondecreasing $l, \alpha \in \textnormal{Cdf}_{\mathbb{R}_{+}}$ wherein $x_0=0, l(\delta) \geq 1, \alpha(\delta) \leq \delta$ on $\mathbb{R}_+$ then
	\begin{align*}
		x(\delta) \leq & \Biggl[\frac { ( \lambda_1 - \lambda_2 ) } { \lambda_1 } \int _ { 0 } ^ { \alpha ( \delta ) } g _ { 3 } ( \tilde{\delta} ) \operatorname { e x \lambda_1 } \Biggl(( \lambda_1 - \lambda_2 ) \int _ { \tilde{\delta} } ^ { \alpha ( \delta ) } \Biggl(\mathfrak{\alpha}^{-1}(\sigma)\Biggl(l^{\prime}\left(\alpha^{-1}(\sigma)\right)+g_1(\sigma)\nonumber\\
		&\qquad+g_2(\sigma)\Biggl)+\frac{1}{\mathfrak{\alpha}^{-1}(\sigma)} \Biggr)d \sigma\Biggr) d \tilde{\delta}\Biggr]^{\frac{1}{\lambda_1-\lambda_2}} \quad \forall \delta \in \mathbb{R}_+.
	\end{align*}
\end{thm}
This work proposes generalized inequalities, expanding on Shakoor's inequalities \cite{Shakoor2023}. Before moving on, we'll go through some of the symbols and notations that will be used in the discussion that follows.
 $\textnormal{Cf}_{\mathbb{R}_+}$ (continuous functions on $\mathbb{R}_+$), $\textnormal{Cdf}_{\mathbb{R}_{+}}$ (continuously differentiable functions on $\mathbb{R}_+$)  and $\mathbb{R}_+, $ indicates $\mathbb{R}_+=[0,\infty).$
\par The subsequent portion of the article  is separated into the following sections: The first section presents some novel retarded nonlinear integral and integro-differential inequalities that generalize the existing inequalities in the literature.  In the second section,  we provide few examples to show the effectiveness of our inequalities in determining and analyzing boundedness and  and global behavior of the solution for nonlinear retarded integral equations of Volterra kind. In the last section, we state some crucial conclusions of this study.

\section{\textbf{Main Results}}
Before proceeding to our main result, we initiate our section with fundamental lemmas, which will come in handy later on.
\begin{lem}\label{i2}
	If $\omega_1, \omega_2\geq 0$ and  $\gamma\geq 1$, then
	\begin{align*}
		(\omega_1+\omega_2)^\gamma \leq 2^{\gamma-1}(\omega_1^\gamma +\omega_2^\gamma).
	\end{align*}
\end{lem}
\begin{lem}(Zhao \cite{zhao})\label{i3}
For any $\omega\geq 0,~\gamma_1\geq \gamma_2\geq 0, \gamma_1\neq 0$, 
\begin{align*} \omega^{\frac{\gamma_2}{\gamma_1}}\leq\frac{\gamma_2}{\gamma_1}\kappa^{\frac{\gamma_2-\gamma_1}{\gamma_1}}\omega+\frac{\gamma_1-\gamma_2}{\gamma_1}\kappa^{\frac{\gamma_2}{\gamma_1}},~\kappa>0.	
\end{align*}
\end{lem}
We begin with a new generalized version of nonlinear retarded integro-differential inequality developed by A Shakoor et al. \cite{Shakoor2023} mentioned in Theorem \ref{i1}.
\begin{thm}\label{t2} 
	If  $\mathfrak{u}, \mathfrak{u}', \Psi_1, \Psi_2, \Psi_3\in \textnormal{Cf}_{\mathbb{R}_+}$ and $a, f\in  \textnormal{Cdf}_{\mathbb{R}_{+}}$ are nondecreasing in nature wherein $a(\delta)\geq 1, f(\delta)\leq \delta ~ (\delta\in \mathbb{R}_+), \mathfrak{u}(0)=0$ are such that 
	\begin{align}\label{2e1}
		(\mathfrak{u}'(\delta))^{\gamma_1}\leq a(\delta)+\int\limits_{0}^{f(\delta)} \Psi_1(\theta)\mathfrak{u}(\theta) d\theta+\int\limits_{0}^{f(\delta)} \Psi_2(\theta)\left(\mathfrak{u}^{\gamma_2}(\theta)+\int\limits_{0}^{\theta} \Psi_3(\xi)\mathfrak{u}^{\gamma_3}(\xi) d\xi\right)^{\frac{1}{\gamma_2}}d\theta
	\end{align}
for $\delta, \gamma_1, \gamma_2, \gamma_3, \in\mathbb{R}_+,$ with $\gamma_1 \geq 1, \gamma_2\geq 2, \gamma_3\geq 1, \gamma_2\neq \gamma_3$ then
\begin{align}\label{2e13}
	\mathfrak{u}(\delta)&\leq \delta~\zeta_2+2^{\frac{1-\gamma_2}{\gamma_2}}\Biggl\{\frac{\gamma_2-\gamma_3}{\gamma_2}\int_{ 0 }^{f(\delta)} \Psi_3(\xi)2^{\frac{\gamma_3-\gamma_2}{\gamma_2}} \exp\Biggl((\gamma_2-\gamma_3)\int_\xi^{f(\delta)}  \Biggl(2^{\frac{\gamma_2-1}{\gamma_2}}f^{-1}(\theta)~\zeta_1a'(f^{-1}(\theta))\nonumber\\
	&+2^{\gamma_2-1}(f^{-1}(\theta))^{\gamma_2-1}~\zeta_2^{\gamma_2}+\frac{1}{\gamma_2}\Psi_3(\theta)2^{\gamma_3-1}(f^{-1}(\theta))^{\gamma_3}~\zeta_2^{\gamma_3}+f^{-1}(\theta)~\zeta_1 \Psi_1(\theta)\nonumber\\
	&~+2^{\frac{\gamma_2-1}{\gamma_2}}(f^{-1}(\theta))^2~\zeta_1\zeta_2\Psi_1(\theta)+2^{\frac{\gamma_2-1}{\gamma_2}}f^{-1}(\theta)~\zeta_1 \Psi_2(\theta)+\frac{1}{f^{-1}(\theta)}\Biggr) \Biggr)\Biggr\}^{\frac{1}{\gamma_2-\gamma_3}},
\end{align}
where $\zeta_1=\frac{1}{\gamma_1}\kappa^{\frac{1-\gamma_1}{\gamma_1}}, \zeta_2=\frac{\gamma_1-1}{\gamma_1}\kappa^{\frac{1}{\gamma_1}}~(\kappa>0).$
\end{thm}
\begin{proof}
If  $\mathfrak{v}(\delta)$  indicates right-hand-side of inequality \eqref{2e1} then $\mathfrak{v}(0)=a(0)$ and from \eqref{2e1}, it is apparent that
\begin{align}\label{2e3}
	\mathfrak{u}'(\delta)\leq \mathfrak{v}^{\frac{1}{\gamma_1}}(\delta)\leq \zeta_1 \mathfrak{v}(\delta)+\zeta_2,~\text{where}~\zeta_1=\frac{1}{\gamma_1}\kappa^{\frac{1-\gamma_1}{\gamma_1}}, \zeta_2=\frac{\gamma_1-1}{\gamma_1}\kappa^{\frac{1}{\gamma_1}},~\text{for any}~\kappa>0.
\end{align}
Further the nondecreasing nature of $\mathfrak{v}(\delta)\geq 0,$ asserts that
\begin{align}\label{2e4}
	\mathfrak{u}(\delta)\leq \delta~\zeta_1 \mathfrak{v}(\delta)+\delta~\zeta_2.
\end{align}
Using \eqref{2e4} and lemma \ref{i2}, we have
\begin{align}\label{2e5}
		\mathfrak{v}'(\delta)&=a'(\delta)+ f'(\delta)\Psi_1(f(\delta))\mathfrak{u}(f(\delta)) + f'(\delta)\Psi_2(f(\delta))\left(\mathfrak{u}^{\gamma_2}(f(\delta))+\int\limits_{0}^{f(\delta)} \Psi_3(\xi)\mathfrak{u}^{\gamma_3}(\xi) d\xi\right)^{\frac{1}{\gamma_2}}\nonumber\\
		&~\leq a'(\delta)+ f'(\delta)\Psi_1(f(\delta))(\delta~\zeta_1 \mathfrak{v}(\delta)+\delta~\zeta_2) + f'(\delta) \Psi_2(f(\delta))\Biggl(2^{\gamma_2-1}\Big(\delta^{\gamma_2}~\zeta_1^{\gamma_2} \mathfrak{v}^{\gamma_2}(\delta)+\delta^{\gamma_2}~\zeta_2^{\gamma_2}\Big)\nonumber\\
		&\qquad+\int\limits_{0}^{f(\delta)} \Psi_3(\xi)(\xi^{\gamma_3}~\zeta_1^{\gamma_3} \mathfrak{v}^{\gamma_3}(\xi)+\xi^{\gamma_3}~\zeta_2^{\gamma_3}\Big) d\xi\Biggr)^{\frac{1}{\gamma_2}}.
\end{align}
Set up  $\mathfrak{w}^{\gamma_2}(\delta)$ as
\begin{align}\label{2e7}
	\mathfrak{w}^{\gamma_2}(\delta)=2^{\gamma_2-1}\Big(\delta^{\gamma_2}~\zeta_1^{\gamma_2} \mathfrak{v}^{\gamma_2}(\delta)+\delta^{\gamma_2}~\zeta_2^{\gamma_2}\Big)+\int\limits_{0}^{f(\delta)} \Psi_3(\xi)2^{\gamma_3-1}\Big(\xi^{\gamma_3}~\zeta_1^{\gamma_3} \mathfrak{v}^{\gamma_3}(\xi)+\xi^{\gamma_3}~\zeta_2^{\gamma_3}\Big) d\xi.
\end{align}
Thus, $\mathfrak{u}(\delta)\leq  \delta~\zeta_1 \mathfrak{v}(\delta)+\delta~\zeta_2\leq 2^{\frac{1-\gamma_2}{\gamma_2}}\mathfrak{w}(\delta)+\delta~\zeta_2$ and $\mathfrak{w}(0)=0$. On differentiating \eqref{2e7}, we see that
\begin{align}\label{2e8}
\gamma_2\mathfrak{w}^{\gamma_2-1}(\delta)\mathfrak{w}'(\delta)&=2^{\gamma_2-1}\Big(\delta^{\gamma_2}~\zeta_1^{\gamma_2} \gamma_2\mathfrak{v}^{\gamma_2-1}(\delta)\mathfrak{v}'(\delta)+\gamma_2\delta^{\gamma_2-1}~\zeta_1^{\gamma_2} \mathfrak{v}^{\gamma_2}(\delta)+\gamma_2\delta^{\gamma_2-1}~\zeta_2^{\gamma_2}\Big)\nonumber\\
&~+ f'(\delta)\Psi_3(f(\delta))2^{\gamma_3-1}\Big(f(\delta)^{\gamma_3}~\zeta_1^{\gamma_3} \mathfrak{v}^{\gamma_3}(f(\delta))+f(\delta)^{\gamma_3}~\zeta_2^{\gamma_3}\Big)\nonumber\\
&\leq 2^{\frac{\gamma_2-1}{\gamma_2}}\delta~\zeta_1 \gamma_2\mathfrak{w}^{\gamma_2-1}(\delta)\mathfrak{v}'(\delta)+\gamma_2\delta^{-1}~ \mathfrak{w}^{\gamma_2}(\delta)+2^{\gamma_2-1}\gamma_2\delta^{\gamma_2-1}~\zeta_2^{\gamma_2}\nonumber\\
&~+f'(\delta)\Psi_3(f(\delta))2^{\frac{\gamma_3-\gamma_2}{\gamma_2}}\mathfrak{w}^{\gamma_3}(\delta)+f'(\delta)\Psi_3(f(\delta))2^{\gamma_3-1}\delta^{\gamma_3}~\zeta_2^{\gamma_3}\nonumber\\
&\leq 2^{\frac{\gamma_2-1}{\gamma_2}}\delta~\zeta_1 \gamma_2\mathfrak{w}^{\gamma_2-1}(\delta)\Biggl\{a'(\delta)+ f'(\delta)\Psi_1(f(\delta))(\delta~\zeta_1 \mathfrak{v}(\delta)+\delta~\zeta_2)\nonumber\\
&~ + f'(\delta) \Psi_2(f(\delta))\mathfrak{w}(\delta)\Biggr\}\nonumber\\
&~+\gamma_2\delta^{-1}~ \mathfrak{w}^{\gamma_2}(\delta)+2^{\gamma_2-1}\gamma_2\delta^{\gamma_2-1}~\zeta_2^{\gamma_2}+f'(\delta)\Psi_3(f(\delta))2^{\frac{\gamma_3-\gamma_2}{\gamma_2}}\mathfrak{w}^{\gamma_3}(\delta)\nonumber\\
&~+f'(\delta)\Psi_3(f(\delta))2^{\gamma_3-1}\delta^{\gamma_3}~\zeta_2^{\gamma_3}.
\end{align}
Further dividing inequality \eqref{2e8} by $\gamma_2\mathfrak{w}^{\gamma_2-1}(\delta)$ with $1\geq \mathfrak{w}^{-1}(\delta)\geq  \mathfrak{w}^{1-\gamma_2}(\delta)$ implies that
\begin{align}\label{2e9}
\mathfrak{w}'(\delta) &\leq 2^{\frac{\gamma_2-1}{\gamma_2}}\delta~\zeta_1 \Biggl\{a'(\delta)+ f'(\delta)\Psi_1(f(\delta))(\delta~\zeta_1 \mathfrak{v}(\delta)+\delta~\zeta_2) + f'(\delta) \Psi_2(f(\delta))\mathfrak{w}(\delta)\Biggr\}\nonumber\\
&~+\delta^{-1}~ \mathfrak{w}(\delta)+2^{\gamma_2-1}\delta^{\gamma_2-1}~\zeta_2^{\gamma_2}\mathfrak{w}^{1-\gamma_2}(\delta)+\frac{1}{\gamma_2}f'(\delta)\Psi_3(f(\delta))2^{\frac{\gamma_3-\gamma_2}{\gamma_2}}\mathfrak{w}^{\gamma_3-\gamma_2+1}(\delta)\nonumber\\
&~+\frac{1}{\gamma_2}f'(\delta)\Psi_3(f(\delta))2^{\gamma_3-1}\delta^{\gamma_3}~\zeta_2^{\gamma_3}\mathfrak{w}^{1-\gamma_2}(\delta)\nonumber\\
&=\Biggl(2^{\frac{\gamma_2-1}{\gamma_2}}\delta~\zeta_1a'(\delta)+2^{\frac{\gamma_2-1}{\gamma_2}}\delta^2~\zeta_1\zeta_2f'(\delta)\Psi_1(f(\delta))+\frac{1}{\gamma_2}f'(\delta)\Psi_3(f(\delta))2^{\gamma_3-1}\delta^{\gamma_3}~\zeta_2^{\gamma_3}\nonumber\\
&~+2^{\gamma_2-1}\delta^{\gamma_2-1}~\zeta_2^{\gamma_2}\Biggr)+\left(\delta~\zeta_1 f'(\delta)\Psi_1(f(\delta))+2^{\frac{\gamma_2-1}{\gamma_2}}\delta~\zeta_1f'(\delta) \Psi_2(f(\delta))+\delta^{-1}\right)\mathfrak{w}(\delta)\nonumber\\
&~+\frac{1}{\gamma_2}f'(\delta)\Psi_3(f(\delta))2^{\frac{\gamma_3-\gamma_2}{\gamma_2}}\mathfrak{w}^{\gamma_3-\gamma_2+1}(\delta).
\end{align}

Suppose $\mathfrak{z}(\delta)=\mathfrak{w}^{\gamma_2-\gamma_3}(\delta)$, thereby, $\mathfrak{z}(0)=0$ and $\mathfrak{w}'(\delta)=\frac{1}{\gamma_2-\gamma_3}\mathfrak{z}'(\delta)\mathfrak{w}^{\gamma_3-\gamma_2+1}(\delta)$. By inserting it in the inequality  \eqref{2e9} and dividing the entire resulting inequality by $\mathfrak{w}^{\gamma_3-\gamma_2+1}(\delta),$ we get
\begin{align}\label{2e11}
	\mathfrak{z}'(\delta) &\leq (\gamma_2-\gamma_3) \Biggl(2^{\frac{\gamma_2-1}{\gamma_2}}\delta~\zeta_1a'(\delta)+2^{\gamma_2-1}\delta^{\gamma_2-1}~\zeta_2^{\gamma_2}+\frac{1}{\gamma_2}f'(\delta)\Psi_3(f(\delta))2^{\gamma_3-1}\delta^{\gamma_3}~\zeta_2^{\gamma_3}\nonumber\\
	&~+\delta~\zeta_1 f'(\delta)\Psi_1(f(\delta))+2^{\frac{\gamma_2-1}{\gamma_2}}\delta^2~\zeta_1\zeta_2f'(\delta)\Psi_1(f(\delta))+2^{\frac{\gamma_2-1}{\gamma_2}}\delta~\zeta_1f'(\delta) \Psi_2(f(\delta))+\delta^{-1}\Biggr)\mathfrak{z}(\delta)\nonumber\\
	&~+\frac{\gamma_2-\gamma_3}{\gamma_2}f'(\delta)\Psi_3(f(\delta))2^{\frac{\gamma_3-\gamma_2}{\gamma_2}}.
\end{align}Integrating inequality \eqref{2e11}, we obtain
\begin{align}\label{2e12}
	\mathfrak{z}(\delta) &\leq \frac{\gamma_2-\gamma_3}{\gamma_2}\int_{ 0 }^{f(\delta)} \Psi_3(\xi)2^{\frac{\gamma_3-\gamma_2}{\gamma_2}} \exp\Biggl((\gamma_2-\gamma_3)\int_\xi^{f(\delta)}  \Biggl(2^{\frac{\gamma_2-1}{\gamma_2}}f^{-1}(\theta)~\zeta_1a'(f^{-1}(\theta))\nonumber\\
	&+2^{\gamma_2-1}(f^{-1}(\theta))^{\gamma_2-1}~\zeta_2^{\gamma_2}+\frac{1}{\gamma_2}\Psi_3(\theta)2^{\gamma_3-1}(f^{-1}(\theta))^{\gamma_3}~\zeta_2^{\gamma_3}+f^{-1}(\theta)~\zeta_1 \Psi_1(\theta)\nonumber\\
	&~+2^{\frac{\gamma_2-1}{\gamma_2}}(f^{-1}(\theta))^2~\zeta_1\zeta_2\Psi_1(\theta)+2^{\frac{\gamma_2-1}{\gamma_2}}f^{-1}(\theta)~\zeta_1 \Psi_2(\theta)+\frac{1}{f^{-1}(\theta)}\Biggr) \Biggr)
\end{align}
Combining this  with $\mathfrak{z}(\delta)=\mathfrak{w}^{\gamma_2-\gamma_3}(\delta)$ and $\mathfrak{u}(\delta)\leq 2^{\frac{1-\gamma_2}{\gamma_2}}\mathfrak{w}(\delta)+\delta~\zeta_2$, we achieve the bound as stated in \eqref{2e13}. 

\end{proof}

\begin{thm}\label{t3}
		If  $\mathfrak{u}, \mathfrak{u}', \Psi_1, \Psi_2, \Psi_3\in \textnormal{Cf}_{\mathbb{R}_+}$ and $a, f\in  \textnormal{Cdf}_{\mathbb{R}_{+}}$ are nondecreasing in nature wherein $a(\delta)\geq 1, f(\delta)\leq \delta ~ (\delta\in \mathbb{R}_+), \mathfrak{u}(0)=0$ are such that 
\begin{align}\label{3e1}
	(\mathfrak{u}'(\delta))^{\gamma_1}\leq a(\delta)+\int\limits_{0}^{f(\delta)} \Psi_1(\theta)\mathfrak{u}(\theta) d\theta+\int\limits_{0}^{f(\delta)} \Psi_2(\theta)\left((\mathfrak{u}'(\theta))^{\gamma_2}+\int\limits_{0}^{\theta} \Psi_3(\xi)\mathfrak{u}(\xi) d\xi\right)^{\frac{1}{\gamma_3}}d\theta
\end{align}
for $\delta, \gamma_1, \gamma_2, \gamma_3 \in\mathbb{R}_+,$ with  $\gamma_1\geq \gamma_2\geq 1, \gamma_3\geq 1$ then
	\begin{align}\label{3e9}
		\mathfrak{u}(\delta)&\leq \delta~\zeta_2+\frac{\zeta_1}{\zeta_3}~\delta \Biggl((\zeta_3~a(0)+\zeta_4)\exp\left(\int_0^{f(\delta)} (f^{-1}\theta)~\zeta_1~\Psi_1(\theta) +\zeta_3~\zeta_5~ \Psi_2(\theta)+\frac{\zeta_1}{\zeta_3}~ (f^{-1}\theta)\Psi_3(\theta) d\theta \right)\nonumber\\
		&\qquad+\int_0^{f(\delta)}\left(\zeta_3a'(f^{-1}(\xi))+\zeta_2\zeta_3\Psi_1(\xi)(f^{-1}(\xi))+\zeta_6~ \Psi_2(\xi)+(f^{-1}\xi)~\zeta_2~\Psi_3(\xi)\right)\nonumber\\
		&\qquad\qquad\times \exp\left(\int_\xi^{f(\delta)} (f^{-1}\theta)~\zeta_1~\Psi_1(\theta) +\zeta_3~\zeta_5~ \Psi_2(\theta)+\frac{\zeta_1}{\zeta_3}~ (f^{-1}\theta)\Psi_3(\theta) d\theta \right)d\xi\Biggr),
	\end{align}
	where $\zeta_1=\frac{1}{\gamma_1}\kappa^{\frac{1-\gamma_1}{\gamma_1}},  \zeta_2=\frac{\gamma_1-1}{\gamma_1}\kappa^{\frac{1}{\gamma_1}},  \zeta_3=\frac{\gamma_2}{\gamma_1}\kappa^{\frac{\gamma_2-\gamma_1}{\gamma_1}},  \zeta_4=\frac{\gamma_1-\gamma_2}{\gamma_1}\kappa^{\frac{\gamma_2}{\gamma_1}},  \zeta_5=\frac{1}{\gamma_3}\kappa^{\frac{1-\gamma_3}{\gamma_3}},  \zeta_6=\frac{\gamma_3-1}{\gamma_1}\kappa^{\frac{1}{\gamma_3}}$ ($\kappa>0$).
\end{thm}
\begin{proof}
	If the right-hand-side of inequality \eqref{2e1} is substituted as $\mathfrak{v}(\delta)$  then $\mathfrak{v}(0)=a(0)$ and thus from lemma \ref{i3}
	\begin{align}\label{3e3}
		\mathfrak{u}'(\delta)\leq \mathfrak{v}^{\frac{1}{\gamma_1}}(\delta)\leq \zeta_1 \mathfrak{v}(\delta)+\zeta_2.
	\end{align}
However, the nondecreasing nature of $\mathfrak{v}(\delta)\geq 0$ gives
	\begin{align}\label{3e4}
		\mathfrak{u}(\delta)\leq \delta~\zeta_1 \mathfrak{v}(\delta)+\delta~\zeta_2.
	\end{align}
	Using \eqref{3e4}, we have
	\begin{align}\label{3e5}
		\mathfrak{v}'(\delta)&=a'(\delta)+ f'(\delta)\Psi_1(f(\delta))\mathfrak{u}(f(\delta)) + f'(\delta)\Psi_2(f(\delta))\left((\mathfrak{u}'(f(\delta)))^{\gamma_2}+\int\limits_{0}^{f(\delta)} \Psi_3(\xi)\mathfrak{u}(\xi) d\xi\right)^{\frac{1}{\gamma_3}}\nonumber\\
		&~\leq a'(\delta)+ f'(\delta)\Psi_1(f(\delta))(\delta~\zeta_1 \mathfrak{v}(\delta)+\delta~\zeta_2) + f'(\delta) \Psi_2(f(\delta))\nonumber\\
		&\qquad\qquad\qquad\qquad\times\Biggl( \mathfrak{v}^{\frac{\gamma_2}{\gamma_1}}(\delta)+\int\limits_{0}^{f(\delta)} \Psi_3(\xi)(\xi~\zeta_1 \mathfrak{v}(\xi)+\xi~\zeta_2) d\xi\Biggr)^{\frac{1}{\gamma_3}}\nonumber\\
		&~\leq a'(\delta)+f'(\delta)\Psi_1(f(\delta))\delta~\zeta_1 \mathfrak{v}(\delta)+f'(\delta)\Psi_1(f(\delta))\delta~\zeta_2+f'(\delta) \Psi_2(f(\delta))\mathfrak{z}^{\frac{1}{\gamma_3}}(\delta),
	\end{align}
where $\zeta_3=\frac{\gamma_2}{\gamma_1}\kappa^{\frac{\gamma_2-\gamma_1}{\gamma_1}}, \zeta_4=\frac{\gamma_1-\gamma_2}{\gamma_1}\kappa^{\frac{\gamma_2}{\gamma_1}}$ ($\kappa>0$) and 
\begin{align}\label{3e6}
\mathfrak{z}(\delta)= \zeta_3 \mathfrak{v}(\delta)+\zeta_4+\int\limits_{0}^{f(\delta)} \Psi_3(\xi)(\xi~\zeta_1 \mathfrak{v}(\xi)+\xi~\zeta_2) d\xi.
\end{align} 
Because $\mathfrak{u}(\delta)\leq \delta~\zeta_1 \mathfrak{v}(\delta)+\delta~\zeta_2\leq \delta~\frac{\zeta_1}{\zeta_3} \mathfrak{z}(\delta)+\delta~{\zeta_2}$ according to \eqref{3e6} and also $\mathfrak{z}^{\frac{1}{\gamma_3}}(\delta)\leq \zeta_5~\mathfrak{z}(\delta)+\zeta_6$, thus
\begin{align}\label{3e7}
\mathfrak{z}'(\delta)&= \zeta_3 \mathfrak{v}'(\delta)+f'(\delta)\Psi_3(f(\delta))(f(\delta)~\zeta_1 \mathfrak{v}(f(\delta))+f(\delta)~\zeta_2)\nonumber\\
&~\leq \zeta_3\Big(a'(\delta)+f'(\delta)\Psi_1(f(\delta))\delta~\zeta_1 \mathfrak{v}(\delta)+f'(\delta)\Psi_1(f(\delta))\delta~\zeta_2+f'(\delta) \Psi_2(f(\delta))\mathfrak{z}^{\frac{1}{\gamma_3}}(\delta)\Big)\nonumber\\
&\qquad+f'(\delta)\Psi_3(f(\delta)) \delta~\zeta_1 \mathfrak{v}(\delta)+f'(\delta)\Psi_3(f(\delta))\delta~\zeta_2\nonumber\\
&=\Big(\zeta_3a'(\delta)+\zeta_2\zeta_3f'(\delta)\Psi_1(f(\delta))\delta+\zeta_6~f'(\delta) \Psi_2(f(\delta))+\delta~\zeta_2~f'(\delta)\Psi_3(f(\delta))\Big)\nonumber\\
&\qquad+\Big(\delta~\zeta_1~f'(\delta)\Psi_1(f(\delta)) +\zeta_3\zeta_5~f'(\delta) \Psi_2(f(\delta))+\frac{\zeta_1}{\zeta_3}~ \delta~f'(\delta)\Psi_3(f(\delta)) \Big)\mathfrak{z}(\delta).
\end{align} 
where $\zeta_5=\frac{1}{\gamma_3}\kappa^{\frac{1-\gamma_3}{\gamma_3}}, \zeta_6=\frac{\gamma_3-1}{\gamma_1}\kappa^{\frac{1}{\gamma_3}}$, for any $\kappa>0$. Integrating inequality \eqref{3e7} from $0$ to $\delta,$ consequently, 
\begin{align}\label{3e8}
	\mathfrak{z}(\delta)&\leq (\zeta_3~a(0)+\zeta_4)\exp\left(\int_0^{f(\delta)} (f^{-1}\theta)~\zeta_1~\Psi_1(\theta) +\zeta_3~\zeta_5~ \Psi_2(\theta)+\frac{\zeta_1}{\zeta_3}~ (f^{-1}\theta)\Psi_3(\theta) d\theta \right)\nonumber\\
	&\qquad+\int_0^{f(\delta)}\left(\zeta_3a'(f^{-1}\xi)+\zeta_2\zeta_3\Psi_1(\xi)(f^{-1}\xi)+\zeta_6~ \Psi_2(\xi)+(f^{-1}\xi)~\zeta_2~\Psi_3(\xi)\right)\nonumber\\
	&\qquad\qquad\times \exp\left(\int_\xi^{f(\delta)} (f^{-1}\theta)~\zeta_1~\Psi_1(\theta) +\zeta_3~\zeta_5~ \Psi_2(\theta)+\frac{\zeta_1}{\zeta_3}~ (f^{-1}\theta)\Psi_3(\theta) d\theta \right)d\xi.
\end{align}
Combining the bound obtained on $\mathfrak{z}(\delta)$ in \eqref{3e8} with \eqref{3e4} and using $\mathfrak{v}(\delta)\leq \frac{1}{\zeta_3}~\mathfrak{z}(\delta),$ we arrive at the bound in \eqref{3e9}.
\end{proof}

\begin{rem}
	The integro-differential inequality of A Shakoor et al. \cite{Shakoor2023} can be produced by allowing $\gamma_1=\gamma_2=1$ and $\gamma_3=p.$
\end{rem}

\begin{thm}\label{t1}
			If  $\mathfrak{u}, \mathfrak{u}', \Psi_1, \Psi_2, \Psi_3\in \textnormal{Cf}_{\mathbb{R}_+}$ and $a, f\in  \textnormal{Cdf}_{\mathbb{R}_{+}}$ are nondecreasing in nature wherein $a(\delta)\geq 1, f(\delta)\leq \delta ~ (\delta\in \mathbb{R}_+)$ are such that 
\begin{align}\label{e1}
	\mathfrak{u}^{\gamma_1}(\delta)\leq \left(a(\delta)+\int\limits_{0}^{f(\delta)} \Psi_1(\theta)\mathfrak{u}(\theta) d\theta+\int\limits_{0}^{f(\delta)} \Psi_2(\theta)\left(\mathfrak{u}^{\gamma_2}(\theta)+\int\limits_{0}^{\theta} \Psi_3(\xi)\mathfrak{u}^{\gamma_3}(\xi) d\xi\right)^{\frac{1}{\gamma_2}}d\theta\right)^{\gamma_4}
\end{align}
	for  $\delta, \gamma_1, \gamma_2, \gamma_3, \gamma_4\in\mathbb{R}_+$ with $\gamma_1\geq \gamma_4 > 0, \gamma_2>\gamma_3\geq 0$ then
\begin{align}\label{e18}
	\mathfrak{u}(\delta)&\leq \Biggl\{(\zeta_7+\zeta_8a)^{\gamma_2-\gamma_3}(0) \exp\left(\int_0^{f(\delta)} \zeta_8[\gamma_2-\gamma_3]\Big(a'(f^{-1}\theta)+ \Psi_1(\theta)+ \Psi_2(\theta)\Big)d\theta\right)\nonumber\\
	&~+\int_0^{f(\delta)} \frac{\gamma_2-\gamma_3}{\gamma_2}\Psi_3(\xi)\exp\left(\int_0^{f(\delta)} \zeta_8[\gamma_2-\gamma_3]\Big(a'(f^{-1}\theta)+ \Psi_1(\theta)+ \Psi_2(\theta)\Big)d\theta\right)d\xi\Biggr\}^{\frac{1}{\gamma_2-\gamma_3}},
\end{align}
where $\zeta_7=\frac{\gamma_1-\gamma_4}{\gamma_1}\kappa^{\frac{\gamma_4}{\gamma_1}}$ and $\zeta_8=\frac{\gamma_4}{\gamma_1}\kappa^{\frac{\gamma_4-\gamma_1}{\gamma_1}} $ ($\kappa>0$).
\end{thm}
\begin{proof}
The inequality \eqref{e1} can be rephrased to the form,
\begin{align}\label{e2}
	\mathfrak{u}(\delta)&\leq \left(a(\delta)+\int\limits_{0}^{f(\delta)} \Psi_1(\theta)\mathfrak{u}(\theta) d\theta+\int\limits_{0}^{f(\delta)} \Psi_2(\theta)\left(\mathfrak{u}^{\gamma_2}(\theta)+\int\limits_{0}^{\theta} \Psi_3(\xi)\mathfrak{u}^{\gamma_3}(\xi) d\xi\right)^{\frac{1}{\gamma_2}}d\theta\right)^{\frac{\gamma_4}{\gamma_1}}\nonumber\\
	&\leq \zeta_7+\zeta_8a(\delta)+\int\limits_{0}^{f(\delta)} \zeta_8\Psi_1(\theta)\mathfrak{u}(\theta) d\theta+\int\limits_{0}^{f(\delta)} \zeta_8\Psi_2(\theta)\left(\mathfrak{u}^{\gamma_2}(\theta)+\int\limits_{0}^{\theta} \Psi_3(\xi)\mathfrak{u}^{\gamma_3}(\xi) d\xi\right)^{\frac{1}{\gamma_2}}d\theta,
\end{align}
where $\zeta_7=\frac{\gamma_1-\gamma_4}{\gamma_1}\kappa^{\frac{\gamma_4}{\gamma_1}}$ and $\zeta_8=\frac{\gamma_4}{\gamma_1}\kappa^{\frac{\gamma_4-\gamma_1}{\gamma_1}}$  ($\kappa>0$).
If  $\mathfrak{v}(\delta)$  indicates right-hand-side of inequality \eqref{e2}, then $\mathfrak{u}(\delta)\leq \mathfrak{v}(\delta)$ with $\mathfrak{v}(0)=\zeta_7+\zeta_8a(0),$ and thus $	\mathfrak{u}(f(\delta))\leq \mathfrak{v}(f(\delta))\leq \mathfrak{v}(\delta)$ due to nondecreasing nature of $\mathfrak{v}(\delta)$. Further,
\begin{align}\label{e6}
	\mathfrak{v}'(\delta)&=\zeta_8a'(\delta)+ \zeta_8f'(\delta)\Psi_1(f(\delta))\mathfrak{u}(f(\delta)) + f'(\delta)\zeta_8\Psi_2(f(\delta))\left(\mathfrak{u}^{\gamma_2}(f(\delta))+\int\limits_{0}^{f(\delta)} \Psi_3(\xi)\mathfrak{u}^{\gamma_3}(\xi) d\xi\right)^{\frac{1}{\gamma_2}}\nonumber\\
	&\leq \zeta_8a'(\delta)+ \zeta_8f'(\delta)\Psi_1(f(\delta))\mathfrak{v}(\delta) + \zeta_8f'(\delta)\Psi_2(f(\delta))\left(\mathfrak{v}^{\gamma_2}(\delta)+\int\limits_{0}^{f(\delta)} \Psi_3(\xi)\mathfrak{v}^{\gamma_3}(\xi) d\xi\right)^{\frac{1}{\gamma_2}}\nonumber\\
	&\leq \zeta_8a'(\delta)+ \zeta_8f'(\delta)\Psi_1(f(\delta))\mathfrak{v}(\delta) + \zeta_8f'(\delta)\Psi_2(f(\delta))\mathfrak{w}(\delta),
\end{align}
where
\begin{align}\label{e7}
\mathfrak{w}(\delta)=\left(\mathfrak{v}^{\gamma_2}(\delta)+\int\limits_{0}^{f(\delta)} \Psi_3(\xi)\mathfrak{v}^{\gamma_3}(\xi) d\xi\right)^{\frac{1}{\gamma_2}}~~\textnormal{i.e.}~~\mathfrak{w}^{\gamma_2}(\delta)=\mathfrak{v}^{\gamma_2}(\delta)+\int\limits_{0}^{f(\delta)} \Psi_3(\xi)\mathfrak{v}^{\gamma_3}(\xi) d\xi.
\end{align}
The equation \eqref{e7} provides that $\mathfrak{w}(0)=\mathfrak{v}(0)=\zeta_7+\zeta_8a(0), \mathfrak{v}(\delta)\leq \mathfrak{w}(\delta),$ and therein
\begin{align}\label{e10}
\gamma_2\mathfrak{w}^{\gamma_2-1}(\delta)\mathfrak{w}'(\delta)&=\gamma_2	\mathfrak{v}^{\gamma_2-1}(\delta)\mathfrak{v}'(\delta)+f'(\delta)\Psi_3(f(\delta))\mathfrak{v}^{\gamma_3}(f(\delta))\nonumber\\
&~\leq\gamma_2	\mathfrak{w}^{\gamma_2-1}(\delta)\mathfrak{v}'(\delta)+f'(\delta)\Psi_3(f(\delta))\mathfrak{w}^{\gamma_3}(f(\delta))\nonumber\\
&~\leq \gamma_2	\mathfrak{w}^{\gamma_2-1}(\delta)\Big(\zeta_8a'(\delta)+ \zeta_8f'(\delta)\Psi_1(f(\delta))\mathfrak{v}(\delta) + \zeta_8f'(\delta)\Psi_2(f(\delta))\mathfrak{w}(\delta)\Big)\nonumber\\
&\qquad+f'(\delta)\Psi_3(f(\delta))\mathfrak{w}^{\gamma_3}(f(\delta)).
\end{align}
Nextly, division of both sides of \eqref{e10} by $\gamma_2	\mathfrak{w}^{\gamma_2-1}(\delta)$ provides the estimate
\begin{align}\label{e11}
	\mathfrak{w}'(\delta)&\leq \zeta_8a'(\delta)+ \zeta_8f'(\delta)\Psi_1(f(\delta))\mathfrak{v}(\delta) + \zeta_8f'(\delta)\Psi_2(f(\delta))\mathfrak{w}(\delta)\nonumber\\
	&\qquad\qquad\qquad\qquad\qquad\qquad+\frac{1}{\gamma_2}f'(\delta)\Psi_3(f(\delta))\mathfrak{w}^{\gamma_3-\gamma_2+1}(\delta).
\end{align}
Consider $\mathfrak{z}(\delta)=\mathfrak{w}^{{\gamma_2-\gamma_3}}(\delta),$ so $\mathfrak{z}(0)=(\zeta_7+\zeta_8a(0))^{\gamma_2-\gamma_3}(0)$ and
\begin{align}\label{e12}
	\mathfrak{w}'(\delta)=\frac{1}{\gamma_2-\gamma_3}\mathfrak{w}^{1+\gamma_3-\gamma_2}(\delta)\mathfrak{z}'(\delta)
\end{align}
Combining \eqref{e11} and \eqref{e12}, we obtain
\begin{align}\label{e13}
\frac{1}{\gamma_2-\gamma_3}\mathfrak{w}^{1+\gamma_3-\gamma_2}(\delta)\mathfrak{z}'(\delta)&\leq \zeta_8a'(\delta)+ \zeta_8f'(\delta)\Psi_1(f(\delta))\mathfrak{w}(\delta) + \zeta_8f'(\delta)\Psi_2(f(\delta))\mathfrak{w}(\delta)\nonumber\\
&\qquad+\frac{1}{\gamma_2}f'(\delta)\Psi_3(f(\delta))\mathfrak{w}^{\gamma_3-\gamma_2+1}(\delta).
\end{align}
Further, utilizing $\mathfrak{w}(\delta)\geq 1$ as $a(\delta)\geq 1,$ a division of inequality \eqref{e13} by $\mathfrak{w}^{1+\gamma_3-\gamma_2}(\delta)$ gives
\begin{align}\label{e14}
	\frac{1}{\gamma_2-\gamma_3}\mathfrak{z}'(\delta)&\leq \zeta_8a'(\delta)\mathfrak{w}^{-1-\gamma_3+\gamma_2}(\delta)+ \zeta_8f'(\delta)\Psi_1(f(\delta))\mathfrak{w}^{-\gamma_3+\gamma_2}(\delta) \nonumber\\
	&\qquad+ \zeta_8f'(\delta)\Psi_2(f(\delta))\mathfrak{w}^{-\gamma_3+\gamma_2}(\delta)+\frac{1}{\gamma_2}f'(\delta)\Psi_3(f(\delta))\nonumber\\
	&=\zeta_8a'(\delta)\mathfrak{w}^{-1}(\delta)\mathfrak{w}^{-\gamma_3+\gamma_2}(\delta)+ \zeta_8f'(\delta)\Psi_1(f(\delta))\mathfrak{w}^{-\gamma_3+\gamma_2}(\delta) \nonumber\\
	&\qquad+ \zeta_8f'(\delta)\Psi_2(f(\delta))\mathfrak{w}^{-\gamma_3+\gamma_2}(\delta)+\frac{1}{\gamma_2}f'(\delta)\Psi_3(f(\delta))\nonumber\\
	&= \zeta_8a'(\delta)\mathfrak{w}^{-1}(\delta)\mathfrak{z}(\delta)+ \zeta_8f'(\delta)\Psi_1(f(\delta))\mathfrak{z}(\delta) \nonumber\\
	&\qquad+ \zeta_8f'(\delta)\Psi_2(f(\delta))\mathfrak{z}(\delta)+\frac{1}{\gamma_2}f'(\delta)\Psi_3(f(\delta))\nonumber\\
	&\leq \zeta_8\Big(a'(\delta)+ f'(\delta)\Psi_1(f(\delta))+f'(\delta)\Psi_2(f(\delta))\Big)\mathfrak{z}(\delta)+\frac{1}{\gamma_2}f'(\delta)\Psi_3(f(\delta))
\end{align}
Integrating inequality \eqref{e14} from $0$ to $\delta,$ we find that
\begin{align}\label{e16}
	\mathfrak{z}(\delta)&\leq (\zeta_7+\zeta_8a(0))^{\gamma_2-\gamma_3}(0) \exp\left(\int_0^{f(\delta)} \zeta_8[\gamma_2-\gamma_3]\Big(a'(f^{-1}\sigma)+ \Psi_1(\sigma)+ \Psi_2(\sigma)\Big)d\sigma\right)\nonumber\\
	&\qquad+\int_0^{f(\delta)} \frac{\gamma_2-\gamma_3}{\gamma_2}\Psi_3(\lambda)\exp\left(\int_0^{f(\delta)} \zeta_8[\gamma_2-\gamma_3]\Big(a'(f^{-1}\sigma)+ \Psi_1(\sigma)+ \Psi_2(\sigma)\Big)d\sigma\right)d\lambda.
\end{align}
Thus from \eqref{e16}, $\mathfrak{w}(\delta)\geq \mathfrak{v}(\delta)\geq \mathfrak{u}(\delta),$ and using definition of $\mathfrak{z}(\delta),$ we achieve the bound as stated in \eqref{e18}.
\end{proof}

\begin{rem}
	Through alteration in the initial assumptions of the Theorem \ref{t1}, we come up with the following widely recognized inequalities.
	\begin{enumerate}
		\item If we set $\gamma_1=1=\gamma_4$, we retrieve the most latest nonlinear retarded integral inequality developed by A Shaknoor et al. (Theorem 2.1 \cite{Shakoor2023}).
		\item The renowned inequality of Gronwall and Bellman \cite{Pachpatte1998} can be acquired if we consider $a(\delta)=c$ for some $c\in\mathbb{R}_+, \Psi_2(\delta)=0, \gamma_1=1=\gamma_4$ and $\mathfrak{a}(\delta)=\delta.$
		\item If we set up the assumptions as $\gamma_1=1=\gamma_4, a(\delta)=c\in\mathbb{R}_+, \Psi_1(\delta)= 0,$ and $f(\delta)=\delta$ then inequality proved above shrinks to Theorem 2.3 \cite{Bainov1992}.
	\end{enumerate}
\end{rem}

\begin{thm}\label{t4}
	Consider $\mathfrak{u}, \Psi_1, \Psi_2, \Psi_3, \Psi_4, \Psi_5, \Psi_6\in \textnormal{Cf}_{\mathbb{R}_+},$ and let $a, f\in  \textnormal{Cdf}_{\mathbb{R}_{+}}$ be nondecreasing in nature wherein $a(\delta)\geq 1, f(\delta)\leq \delta (\delta\in \mathbb{R}_+)$ such that
	\begin{align}\label{4e1}
		\mathfrak{u}^{\gamma_1}(\delta)&\leq a(\delta)+\int\limits_{0}^{f(\delta)} (\Psi_1(\theta)\mathfrak{u}(\theta)+\Psi_2(\theta)) d\theta+\int\limits_{0}^{f(\delta)} \Biggl\{\Psi_3(\theta)\Biggl(\mathfrak{u}^{\gamma_1}(\theta)\nonumber\\
		&\qquad+\int\limits_{0}^{\theta} (\Psi_4(\xi)\mathfrak{u}^{\gamma_2}(\xi)+\Psi_5(\xi)) d\xi\Biggr)^{\frac{1}{\gamma_1}}+\Psi_6(\theta)\Biggr\}d\theta
	\end{align}
	for  $\delta, \gamma_1, \gamma_2\in\mathbb{R}_+$ with $\gamma_1\geq \gamma_2 \geq 1$ then
	\begin{align}\label{4e7}
		\mathfrak{u}(\delta)&\leq \Biggl\{a(0) \exp\Biggl(\int_0^{f(\delta)}\zeta_1~\Psi_1(\theta)+\zeta_1~\Psi_3(\theta)+\zeta_3~\Psi_4(\theta)d\theta\Biggr)+\int_0^{f(\delta)} \Big(a'(f^{-1}\xi)+\zeta_2~\Psi_1(\xi)\nonumber\\
		&\qquad+\Psi_2(\xi)+\zeta_2~\Psi_3(\xi)+\Psi_6(\xi)+\zeta_4\Psi_4(\xi)+\Psi_5(\xi)\Big)\times \exp\Biggl(\int_\xi^{f(\delta)}\zeta_1~\Psi_1(\theta)+\zeta_1~\Psi_3(\theta)\nonumber\\
		&\qquad+\zeta_3~\Psi_4(\theta)d\theta\Biggr)d\xi\Biggr\}^{\frac{1}{p}},
	\end{align}
where $\zeta_1, \zeta_2, \zeta_3, \zeta_4$ are as in Theorem \ref{t3}.
\end{thm}
\begin{proof}
	If the right-hand-side of inequality \eqref{4e1} is substituted as $\mathfrak{v}(\delta)$  then $	\mathfrak{u}^{\gamma_1}(\delta)\leq\mathfrak{v}(\delta),~ \mathfrak{v}(0)=a(0)$, and so from lemma \ref{i3},
\begin{align}\label{4e3}
	\mathfrak{v}'(\delta)&=a'(\delta)+ f'(\delta)(\Psi_1(f(\delta))\mathfrak{u}(f(\delta))+\Psi_2(f(\delta)))+ f'(\delta)\Biggl\{\Psi_3(f(\delta))\nonumber\\
	&\qquad\qquad\qquad\Biggl(\mathfrak{u}^{\gamma_1}(f(\delta))+\int\limits_{0}^{f(\delta)} (\Psi_4(\xi)\mathfrak{u}^{\gamma_2}(\xi)+\Psi_5(\xi)) d\xi\Biggr)^{\frac{1}{\gamma_1}}+\Psi_6(f(\delta))\Biggr\}\nonumber\\
	&~\leq a'(\delta)+ f'(\delta)(\Psi_1(f(\delta))\mathfrak{v}^{\frac{1}{\gamma_1}}(\delta)+\Psi_2(f(\delta)))+ f'(\delta)\Biggl\{\Psi_3(f(\delta))\nonumber\\
	&\qquad\qquad\qquad\Biggl(\mathfrak{v}(\delta)+\int\limits_{0}^{f(\delta)} (\Psi_4(\xi)\mathfrak{v}^{\frac{\gamma_2}{\gamma_1}}(\xi)+\Psi_5(\xi)) d\xi\Biggr)^{\frac{1}{\gamma_1}}+\Psi_6(f(\delta))\Biggr\}\nonumber\\
	&~\leq a'(\delta)+ f'(\delta)(\Psi_1(f(\delta))(\zeta_1~\mathfrak{v}(\delta)+\zeta_2)+\Psi_2(f(\delta)))+f'(\delta)\Big(\Psi_3(f(\delta))\nonumber\\
	&\qquad\qquad\qquad\times(\zeta_1~\mathfrak{w}(\delta)+\zeta_2)+\Psi_6(f(\delta))\Big),
\end{align}
where
\begin{align}\label{4e4}
	\mathfrak{w}(\delta)=\mathfrak{v}(\delta)+\int\limits_{0}^{f(\delta)} (\Psi_4(\xi)\mathfrak{v}^{\frac{\gamma_2}{\gamma_1}}(\xi)+\Psi_5(\xi)) d\xi,\quad \mathfrak{w}(0)=a(0),~\textnormal{and}\quad \mathfrak{v}(\delta)\leq \mathfrak{w}(\delta).
\end{align}
On differentiating $\mathfrak{w}(\delta)$ and using \eqref{4e4}, we find that
\begin{align}\label{4e5}
	\mathfrak{w}'(\delta)&=\mathfrak{v}'(\delta)+f'(\delta) (\Psi_4(f(\delta))\mathfrak{v}^{\frac{\gamma_2}{\gamma_1}}(f(\delta))+\Psi_5(f(\delta)))\nonumber\\
	&~\leq a'(\delta)+ f'(\delta)(\Psi_1(f(\delta))(\zeta_1~\mathfrak{v}(\delta)+\zeta_2)+\Psi_2(f(\delta)))+ f'(\delta)\Big(\Psi_3(f(\delta))(\zeta_1~\mathfrak{w}(\delta)+\zeta_2)\nonumber\\
	&\qquad+\Psi_6(f(\delta))\Big)+f'(\delta) (\Psi_4(f(\delta))(\zeta_3\mathfrak{v}(\delta)+\zeta_4)+\Psi_5(f(\delta)))\nonumber\\
	&~=\Big(\zeta_1~f'(\delta)\Psi_1(f(\delta))+\zeta_1~f'(\delta)\Psi_3(f(\delta))+\zeta_3~f'(\delta)\Psi_4(f(\delta))\Big)\mathfrak{w}(\delta)\nonumber\\
	&\qquad\qquad+\Big(a'(\delta)+\zeta_2~f'(\delta)\Psi_1(f(\delta))+f'(\delta)\Psi_2(f(\delta))+\zeta_2~f'(\delta)\Psi_3(f(\delta))\nonumber\\
	&\qquad\qquad\qquad+f'(\delta)\Psi_6(f(\delta))+\zeta_4~f'(\delta)\Psi_4(f(\delta))+f'(\delta)\Psi_5(f(\delta))\Big).
\end{align}
Integrating inequality \eqref{4e5} from $0$ to $\delta,$ we achieve that
\begin{align}\label{4e6}
\mathfrak{w}(\delta)&\leq a(0) \exp\Biggl(\int\limits_0^{f(\delta)}\zeta_1~\Psi_1(\theta)+\zeta_1~\Psi_3(\theta)+\zeta_3~\Psi_4(\theta)d\theta\Biggr)+\int\limits_0^{f(\delta)} \Big(a'(f^{-1}\xi)+\zeta_2~\Psi_1(\xi)+\Psi_2(\xi)\nonumber\\
&~+\zeta_2~\Psi_3(\xi)+\Psi_6(\xi)+\zeta_4\Psi_4(\xi)+\Psi_5(\xi)\Big)\times \exp\Biggl(\int\limits_\xi^{f(\delta)}\zeta_1~\Psi_1(\theta)+\zeta_1~\Psi_3(\theta)+\zeta_3~\Psi_4(\theta)d\theta\Biggr)d\xi.
\end{align}
Using $\mathfrak{u}^{\gamma_1}(\delta)\leq\mathfrak{v}(\delta)$, \eqref{4e4} and \eqref{4e6}, we find the estimate as stated in \eqref{4e7}.
\end{proof}

\begin{rem} We note that this result reduces to certain recent and well-known integral inequalities under an appropriate set of assumptions as below:
	\begin{enumerate}
		\item If we insert $\Psi_2(\delta)=\Psi_5(\delta)=\Psi_6(\delta)=0$ then this inequality is reduced to Theorem 2.4 \cite{Shakoor2023}.
		\item One can achieve the well-known inequality due to Gronwall and Bellman \cite{Pachpatte1998} from Theorem \ref{t4} if it is assumed that  $a(t)=c$ for some $c\in\mathbb{R}_+, \Psi_2(\delta)=\Psi_3(\delta)=0, \gamma_1=1$ and $f(\delta)=\delta.$
		\item When we set $a(\delta)=c~ (c\in\mathbb{R}_+), \Psi_1(\delta)=0, \Psi_2(\delta)=0, \Psi_5(\delta)=0, \Psi_6(\delta)=0,$ and $f(\delta)=\delta$, the inequality established and proven in Theorem \ref{t4} changes into the inequality shown in Theorem 2.3 \cite{Bainov1992}.
		\item Substituting $\Psi_1=0, \Psi_2(\delta)=0, \Psi_5(\delta)=0, \Psi_6(\delta)=0, a(\delta)=c, f(\delta)=\delta, \gamma_1=\gamma_2=1 $ in the previous inequality yields the same form as the inequality defined in Pachpatte's Theorem \ref{i6}. 
	\end{enumerate}
\end{rem}

\begin{thm}\label{t5}
	Consider $\mathfrak{u}, \Psi_1, \Psi_2, \Psi_3\in \textnormal{Cf}_{\mathbb{R}_+},$ and let $a, f, \Phi\in  \textnormal{Cdf}_{\mathbb{R}_{+}}$ be nondecreasing in nature wherein $a(\delta)\geq 1, \Phi(\delta)\geq 1, f(\delta)\leq \delta (\delta\in \mathbb{R}_+)$ such that
	\begin{align}\label{5e1}
		\mathfrak{u}^{\gamma_1}(\delta)&\leq \Phi(\delta)\Biggl[a(\delta)+\int\limits_{0}^{f(\delta)}\Psi_1(\theta)\mathfrak{u}(\theta) d\theta+\int\limits_{0}^{f(\delta)} \Psi_2(\theta)\Biggl(\mathfrak{u}^{\gamma_1}(\theta)+\int\limits_{0}^{\theta} \Psi_3(\xi)\mathfrak{u}^{\gamma_2}(\xi) d\xi\Biggr)^{\frac{1}{\gamma_2}}d\theta\Biggr]
	\end{align}
	 for $\delta, \gamma_1, \gamma_2\in\mathbb{R}_+$ such that $\gamma_1\geq \gamma_2 \geq 1$ then
\begin{align}\label{5e7}
	\mathfrak{u}(\delta)&\leq \Biggl\{\Phi(0)a(0) \exp\Biggl(\int_0^{f(\delta)}\Big[\Phi'(f^{-1}(\theta))\Phi^{-1}(f^{-1}(\theta))+\zeta_1~\Phi(f^{-1}(\theta))\Psi_1(\theta)+\zeta_9~\Phi(f^{-1}(\theta))\Psi_2(\theta)\nonumber\\
	&~~+\zeta_3~ \Psi_3(\theta)\Big]d\theta\Biggr)+\int_0^{f(\delta)}\Phi(a^{-1}(\xi))a'(a^{-1}(\xi))+\zeta_2~\Phi(a^{-1}(\xi))\Psi_1(\xi)+\zeta_{10}~\Phi(a^{-1}(\xi))\Psi_2(\xi)\nonumber\\
	&~~+\zeta_4~\Psi_3(\xi)\times \Bigg(\exp\Biggl(\int_0^{f(\delta)}\Big[\Phi'(f^{-1}(\theta))\Phi^{-1}(f^{-1}(\theta))+\zeta_1~\Phi(f^{-1}(\theta))\Psi_1(\theta)\nonumber\\
	&~~+\zeta_9~\Phi(f^{-1}(\theta))\Psi_2(\theta)+\zeta_3~ \Psi_3(\theta)\Big]d\theta\Biggr)\Bigg)d\xi\Biggr\}^{\frac{1}{\gamma_1}},
\end{align}
where $\zeta_1, \zeta_2, \zeta_3, \zeta_4$ are as in Theorem \ref{t3} and $\zeta_9=\frac{1}{\gamma_2}\kappa^{\frac{1-\gamma_2}{\gamma_2}},  \zeta_{10}=\frac{\gamma_2-1}{\gamma_2}\kappa^{\frac{1}{\gamma_2}} (\kappa>0)$.
\end{thm}
\begin{proof}
	We begin by substituting right-hand-side of \eqref{5e1} with the function $\mathfrak{v}(\delta)$. It brings us to the conclusions that $\mathfrak{u}^{\gamma_1}(\delta)\leq \mathfrak{v}(\delta)$ and $\mathfrak{v}(0)=\Phi(0)a(0)$. Further derivative of $\mathfrak{v}(\delta),$ and lemmas \ref{i2} and \ref{i3} directs us to
\begin{align}\label{5e3}
\mathfrak{v}'(\delta)&=\Phi'(\delta)\Biggl[a(\delta)+\int\limits_{0}^{f(\delta)}\Psi_1(\theta)\mathfrak{u}(\theta) d\theta+\int\limits_{0}^{f(\delta)} \Psi_2(\theta)\Biggl(\mathfrak{u}^{\gamma_1}(\theta)+\int\limits_{0}^{\theta} \Psi_3(\xi)\mathfrak{u}^{\gamma_2}(\xi) d\xi\Biggr)^{\frac{1}{\gamma_2}}d\theta\Biggr]\nonumber\\
&~+\Phi(\delta)\Biggl[a'(\delta)+f'(\delta)\Psi_1(f(\delta))\mathfrak{u}(f(\delta))+f'(\delta) \Psi_2(f(\delta))\Biggl(\mathfrak{u}^{\gamma_1}(f(\delta))+\int\limits_{0}^{f(\delta)} \Psi_3(\xi)\mathfrak{u}^{\gamma_2}(\xi) d\xi\Biggr)^{\frac{1}{\gamma_2}}\Biggr]\nonumber\\
&\leq \Phi'(\delta)\Phi^{-1}(\delta)\mathfrak{v}(\delta)+\Phi(\delta)\Biggl(a'(\delta)+f'(\delta)\Psi_1(f(\delta))\mathfrak{v}^{\frac{1}{\gamma_1}}(\delta)+f'(\delta) \Psi_2(f(\delta))\mathfrak{w}^{\frac{1}{\gamma_2}}(\delta)\Biggr),\nonumber\\
&\qquad\qquad\textnormal{where}~~\mathfrak{w}(\delta)=\mathfrak{v}(\delta)+\int\limits_{0}^{f(\delta)} \Psi_3(\xi)\mathfrak{v}^{\frac{\gamma_2}{\gamma_1}}(\xi) d\xi\nonumber\\
&\leq \Phi'(\delta)\Phi^{-1}(\delta)\mathfrak{v}(\delta)+\Phi(\delta)a'(\delta)+\Phi(\delta)f'(\delta)\Psi_1(f(\delta))(\zeta_1~\mathfrak{v}(\delta)+\zeta_2)\nonumber\\
&\qquad+\Phi(\delta)f'(\delta) \Psi_2(f(\delta))(\zeta_9~\mathfrak{w}(\delta)+\zeta_{10}).
\end{align}
Since $\mathfrak{v}(\delta)\leq \mathfrak{w}(\delta)$ and $\mathfrak{w}(0)=\Phi(0)a(0),$ from \eqref{5e3}, we have
\begin{align}\label{5e5}
	\mathfrak{w}'(\delta)&=\mathfrak{v}'(\delta)+f'(\delta) \Psi_3(f(\delta))\mathfrak{v}^{\frac{\gamma_2}{\gamma_1}}(f(\delta))\nonumber\\
	&~\leq \Phi'(\delta)\Phi^{-1}(\delta)\mathfrak{w}(\delta)+\Phi(\delta)a'(\delta)+\Phi(\delta)f'(\delta)\Psi_1(f(\delta))(\zeta_1~\mathfrak{w}(\delta)+\zeta_2)\nonumber\\
	&\qquad+\Phi(\delta)f'(\delta) \Psi_2(f(\delta))(\zeta_9~\mathfrak{w}(\delta)+\zeta_{10})+f'(\delta) \Psi_3(f(\delta))(\zeta_3~\mathfrak{w}(\delta)+\zeta_4)\nonumber\\
	&~=\Big(\Phi'(\delta)\Phi^{-1}(\delta)+\zeta_1~\Phi(\delta)f'(\delta)\Psi_1(f(\delta))+\zeta_9~\Phi(\delta)f'(\delta)\Psi_2(f(\delta))+\zeta_3~f'(\delta) \Psi_3(f(\delta))\Big)\mathfrak{w}(\delta)\nonumber\\
	&~+\Big(\Phi(\delta)a'(\delta)+\zeta_2~\Phi(\delta)f'(\delta)\Psi_1(f(\delta))+\zeta_{10}~\Phi(\delta)f'(\delta)\Psi_2(f(\delta))+\zeta_4~f'(\delta) \Psi_3(f(\delta))\Big)
\end{align}
Integrating inequality \eqref{5e5} from $0$ to $\delta,$ we find that
\begin{align}\label{5e6}
	\mathfrak{w}(\delta)&\leq \Phi(0)a(0) \exp\Biggl(\int_0^{f(\delta)}\Big[\Phi'(f^{-1}(\theta))\Phi^{-1}(f^{-1}(\theta))+\zeta_1~\Phi(f^{-1}(\theta))\Psi_1(\theta)+\zeta_9~\Phi(f^{-1}(\theta))\Psi_2(\theta)\nonumber\\
	&~~+\zeta_3~ \Psi_3(\theta)\Big]d\theta\Biggr)+\int_0^{f(\delta)}\Phi(a^{-1}(\xi))a'(a^{-1}(\xi))+\zeta_2~\Phi(a^{-1}(\xi))\Psi_1(\xi)+\zeta_{10}~\Phi(a^{-1}(\xi))\Psi_2(\xi)\nonumber\\
	&~~+\zeta_4~\Psi_3(\xi)\times \Bigg(\exp\Biggl(\int_0^{f(\delta)}\Big[\Phi'(f^{-1}(\theta))\Phi^{-1}(f^{-1}(\theta))+\zeta_1~\Phi(f^{-1}(\theta))\Psi_1(\theta)\nonumber\\
	&~~+\zeta_9~\Phi(f^{-1}(\theta))\Psi_2(\theta)+\zeta_3~ \Psi_3(\theta)\Big]d\theta\Biggr)\Bigg)d\xi.
\end{align}
Thus, using $\mathfrak{u}^{\gamma_1}(\delta)\leq \mathfrak{v}(\delta)\leq \mathfrak{w}(\delta),$ we arrive at the bound as stated in \eqref{5e7}. 
\end{proof}

\begin{rem}Under a suitable set of assumptions, as listed below, we remark that this result simplifies to a few current and well-known integral inequalities.
	\begin{enumerate}
		\item If we set $\Phi(\delta)=1, \delta\in\mathbb{R}_+,$ then above inequality takes the form of inequality due to A Shakoor et al. \cite{Shakoor2023}.
		\item Theorem \ref{t5} is simplified to Gronwall-Bellman inequality \cite{Pachpatte1998} by taking into assumptions that  $\Phi(\delta)=1~ (\delta\in\mathbb{R}_+), a(\delta)=c~ (c\in\mathbb{R}_+),  \Psi_2(\delta)= 0, f(\delta)=\delta,$ and  $\gamma_1=1$.
	\item In particular, Theorem \ref{t5} results into Theorem 2.3 \cite{Bainov1992}, when we choose $\Phi(\delta)=1,$ for $\delta\in\mathbb{R}_+,$ $a(\delta)=c~ (c\in\mathbb{R}_+), \Psi_1(\delta)= 0,$ and $f(\delta)=\delta$.
	\item If we specify the following functions and parameters: $\Phi(\delta)=1$, $a(\delta)=c$, $\Psi_1(\delta)=0$, $f(\delta)=\delta$, and setting $\gamma_1$ and $\gamma_2$ both to 1, then the inequality proven in Theorem \ref{t5} simplifies to Pachpatte's inequality as noted in Theorem \ref{i6}.
	\end{enumerate}

\end{rem}

\section{\textbf{Applications}}
\begin{example}
	\setlength{\belowdisplayskip}{0pt} \setlength{\belowdisplayshortskip}{0pt}
	\setlength{\abovedisplayskip}{0pt} \setlength{\abovedisplayshortskip}{0pt}
Assume an integral equation with nonlinear retardation as
\begin{align}\label{a1e1}
	\mathfrak{u}^5(\delta)\leq \left(\delta+\int\limits_{0}^{\sqrt{\delta}} 2\mathfrak{u}(\theta) d\theta+\int\limits_{0}^{\sqrt{\delta}} 3\left(\mathfrak{u}^{4}(\theta)+\int\limits_{0}^{\theta} \xi\mathfrak{u}^{3}(\xi) d\xi\right)^{\frac{1}{4}}d\theta\right)^{3}.
\end{align}
We can observe that the unknown function $\mathfrak{u}(\delta)$ in \eqref{a1e1} is as stated in Theorem \ref{t1}, thus utilizing Theorem \ref{t1},
\begin{align}\label{a1e2}
	\mathfrak{u}(\delta)&\leq (\zeta_7+\zeta_8\delta)(0) \exp\left(\int_0^{\sqrt{\delta}} 6\zeta_8 d\theta\right)+\int_0^{\sqrt{\delta}} \frac{\xi}{4}\exp\left(\int_0^{\sqrt{\delta}} 6\zeta_8~d\theta\right)d\xi,
\end{align}
where $\zeta_7=\frac{2}{5}\kappa^{\frac{3}{5}}$ and $\zeta_8=\frac{3}{5}\kappa^{\frac{-2}{5}},$ for any $\kappa>0.$ If we let $\kappa=1,$ then
\begin{align}\label{a1e3}
	\mathfrak{u}(\delta)&\leq \Big(\frac{2}{5}+\frac{3}{5}\delta\Big)(0) \exp\left(\int_0^{\sqrt{\delta}} \frac{18}{5} d\theta\right)+\int_0^{\sqrt{\delta}} \frac{\xi}{4}\exp\left(\int_0^{\sqrt{\delta}} \frac{18}{5}~d\theta\right)d\xi\nonumber\\
	&~=\frac{2}{5} \exp\left( \frac{18\sqrt{\delta}}{5} \right)+\int_0^{\sqrt{\delta}} \frac{\xi}{4}\exp\left( \frac{18(-\xi+\sqrt{\delta})}{5}\right)d\xi\nonumber\\
	&~=\frac{2}{5} \exp\left( \frac{18\sqrt{\delta}}{5} \right)+\frac{5 \left(-18 \sqrt{t}+5 e^{\frac{18 \sqrt{t}}{5}}-5\right)}{1296}.
\end{align}
We notice that blow-up does not occur at any point $\delta\in\mathbb{R}_+$, indicating that the solution of \eqref{a1e1} is globally defined.
\end{example}

\begin{example}
	\setlength{\belowdisplayskip}{0pt} \setlength{\belowdisplayshortskip}{0pt}
	\setlength{\abovedisplayskip}{0pt} \setlength{\abovedisplayshortskip}{0pt}
	Assume an integral equation with nonlinear retardation as
		\begin{align}\label{a2e1}
		\mathfrak{u}^3(\delta)&\leq 1+2\delta+\int\limits_{0}^{\delta^{\frac{1}{3}}} (2\mathfrak{u}(\theta)+\theta) d\theta+\int\limits_{0}^{\delta^{\frac{1}{3}}} \Biggl\{5\Biggl(\mathfrak{u}^{3}(\theta)+\int\limits_{0}^{\theta} (7\mathfrak{u}^{2}(\xi)+\xi) d\xi\Biggr)^{\frac{1}{3}}+\theta\Biggr\}d\theta.
	\end{align}
	We notice that the definition of the function $\mathfrak{u}(\delta)$ in equation \eqref{a2e1} is in accordance with Theorem \ref{t4}. The value of $\mathfrak{u}(\delta)$ is thus clearly estimated by applying Theorem \ref{t4} to equation \eqref{a2e1} , and can be represented as follows:
	\begin{align}\label{a2e2}
	\mathfrak{u}(\delta)&\leq \Biggl\{ \exp\Biggl(\int_0^{\delta^{\frac{1}{3}}}(2~\zeta_1+5~\zeta_1+7~\zeta_3) d\theta\Biggr)+\int_0^{\delta^{\frac{1}{3}}} \Big(2+2~\zeta_2+3\xi+5~\zeta_2+7~\zeta_4\Big)\nonumber\\
	&\qquad\times \exp\Biggl(\int_\xi^{\delta^{\frac{1}{3}}}(2~\zeta_1+5~\zeta_1+7~\zeta_3)d\theta\Biggr)d\xi\Biggr\}^{\frac{1}{3}},
\end{align}
where $\zeta_1=\frac{1}{3}\kappa^{\frac{-2}{3}},  \zeta_2=\frac{2}{3}\kappa^{\frac{1}{3}},  \zeta_3=\frac{2}{3}\kappa^{\frac{-1}{3}},  \zeta_4=\frac{1}{3}\kappa^{\frac{2}{3}}$, for any $\kappa>0$. If we set $\kappa=1,$ then we find that
	\begin{align}\label{a2e3}
	\mathfrak{u}(\delta)&\leq \Biggl\{ \exp\Biggl(\int_0^{\delta^{\frac{1}{3}}} 7~ d\theta\Biggr)+\int_0^{\delta^{\frac{1}{3}}} \Big(9+3\xi\Big)\times \exp\Biggl(\int_\xi^{\delta^{\frac{1}{3}}} 7~d\theta\Biggr)d\xi\Biggr\}^{\frac{1}{3}}\nonumber\\
	&~=\Biggl\{ \exp\Biggl(7 \sqrt[3]{\delta}\Biggr)+\int_0^{\delta^{\frac{1}{3}}} \Big(9+3\xi\Big)\times \exp\Biggl(7 (\sqrt[3]{\delta}-\xi)\Biggr)d\xi\Biggr\}^{\frac{1}{3}} \nonumber\\
	&~=\Biggl\{ \exp\Biggl(7 \sqrt[3]{t}\Biggr)+\frac{3}{49} \left(-7 \sqrt[3]{\delta}+22 e^{7 \sqrt[3]{\delta}}-22\right)\Biggr\}^{\frac{1}{3}}.
\end{align}
This plot representation of explicit bound on $\mathfrak{u}(\delta)$ to analyze the bound for blow-up is 
\begin{center}
\includegraphics[scale=0.8]{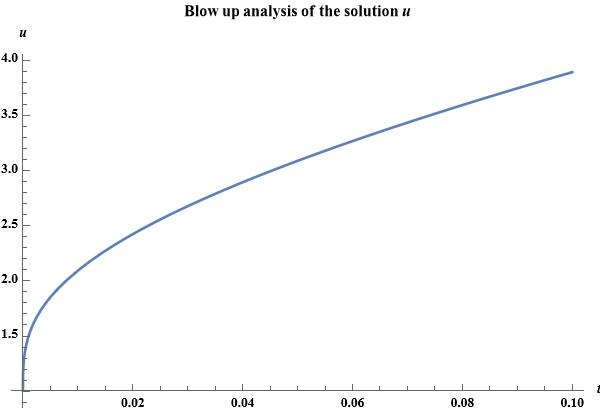}
\end{center}
This plot indicates that the solution does not blow up for any $\delta\in\mathbb{R}_+,$ hence the solution of the equation \eqref{a2e1} is globally defined.
\end{example}
\setlength{\belowdisplayskip}{0pt} \setlength{\belowdisplayshortskip}{0pt}
\setlength{\abovedisplayskip}{0pt} \setlength{\abovedisplayshortskip}{0pt}
\section{\textbf{Conclusions}}
Some novel nonlinear integral and integro-differential inequalities of Gronwall-Bellman-Pachpatte kind are investigated in this work. It demonstrates how a variety of well-known inequalities from both the current literature and the most recent research can be attained by careful choice of parameters. The manuscript then uses the introduced integral inequalities to investigate the existence, uniqueness, stability, boundedness, and asymptotic behavior of solutions to more complicated nonlinear differential and integral equations. Additional important integral and integro-differential problems can be tackled with the help of generalized versions of useful integral inequalities provided by this study.

%\bibliographystyle{ieeetr}
%\bibliography{references}

\end{document}